\documentclass[12pt]{amsart}

\input xy
\usepackage[english]{babel}
\usepackage[latin1]{inputenc}
\usepackage{graphicx}
\usepackage{amsmath}
\usepackage{amssymb}
\usepackage{amscd}
\usepackage{color}
\usepackage{oldgerm}
\usepackage{amsfonts}
\usepackage{newlfont}
\usepackage{longtable}
\usepackage{multirow}
\usepackage{xcolor}

\DeclareOldFontCommand{\rm}{\normalfont\rmfamily}{\mathrm}

\usepackage[all]{xy}

\usepackage{upref}

\def\F{\Bbb F}

\def\N{\Bbb N}

\def\ad{\operatorname{ad}}

\def\Der{\operatorname{Der}}

\def\dim{\operatorname{dim}}

\def\Hom{\operatorname{Hom}}
\def\Ker{\operatorname{Ker}}

\def\Id{\operatorname{Id}}

\def\Span{\operatorname{Span}}

\def\Im{\operatorname{Im}}

\def\g{\frak g}
\def\gl{\frak{gl}}
\def\h{\frak h}

\def\s{\frak{s}}

\def\a{\frak{a}}

\def\p{\frak{p}}

\theoremstyle{plain}\swapnumbers

\newtheorem{Theorem}{Theorem}[section]
\newtheorem{Lemma}[Theorem]{Lemma}
\newtheorem{Prop}[Theorem]{Proposition}

\newtheorem{Cor}[Theorem]{Corollary}

\newtheorem{Remark}[Theorem]{Remark}

\newtheorem{claim}[Theorem]{Claim}

\newtheorem*{Lemma A}{Lemma A}
\newtheorem*{Lema B}{Lemma A}

\newtheorem*{Theorem A}{Theorem A}

\title[Inductive description of quadratic Hom-Lie algebras]{Inductive description of quadratic Hom-Lie algebras with twist maps in the centroid}

\author{R. Garc\'ia-Delgado}

\address{Centro de Investigaci\'on en Matem\'aticas, A.C., 
Unidad M\'erida; Carretera Sierra Papacal Chuburna Puerto Km 5, 97302, Yucat\'an, M\'exico.}
\email{rosendo.garciadelgado@gmail.com,rosendo.garcia@cimat.mx}


\date{\today}

\begin{document}

\maketitle

\begin{abstract}
In this work we give an inductive way to construct quadratic Hom-Lie algebras with twist maps in the centroid. We focus on those Hom-Lie algebras which are not Lie algebras. We prove that a Hom-Lie algebra of this type has trivial center and its twist map is nilpotent. We show that there exists a maximal ideal containing the kernel and the image of the twist map. Then we state an inductive way to construct this type of Hom-Lie algebras, similar to the double extension procedure for Lie algebras, and prove that any indecomposable quadratic Hom-Lie algebra with nilpotent twist map in the centroid, which is not a Lie algebra, can be constructed using this type of double extension. 
\end{abstract}

\section*{Introduction}

A Hom-Lie algebra is a triple $(\g,[\,\cdot\,,\,\cdot\,],T)$, where $\g$ is a vector space over a field $\F$, $[\,\cdot\,,\,\cdot\,]$ is a skew-symmetric bilinear map and $T:\g \to \g$ is a linear map satisfying the \emph{Hom-Lie Jacobi identity}:
$$
[T(x),[y,z]]+[T(y),[z,x]]+[T(z),[x,y]]=0,
$$
for all $x,y,z$ in $\g$. The bilinear map $[\,\cdot\,,\,\cdot\,]$ is called \emph{bracket} and $T$ is called the \emph{twist map}. When the context is clear, we refer to a Hom-Lie algebra by $\g$. The study of Hom-Lie algebras begins in \cite{HLS}, which were motivated by quantum deformations of algebras of vector fields. 
\smallskip

If a Hom-Lie algebra $\g$ admits a non-degenerate and symmetric bilinear form $B:\g \times \g \to \F$, such that $B$ is \emph{invariant}: $B(x,[y,z])=B([x,y],z)$ and $T$ is self adjoint: $B(T(x),y)=B(x,T(y))$, then $\g$ is said to be \emph{quadratic} and $B$ is called an \emph{invariant metric}. A quadratic Hom-Lie algebra is denoted by $(\g,[\,\cdot\,,\,\cdot\,],T,B)$. We recover classical quadratic Lie algebras when $\g$ is a quadratic Hom-Lie algebra for $T=\Id_{\g}$. Quadratic Hom-Lie algebras where the twist map $T$ is a morphism, that is $T([x,y])=[T(x),T(y)]$, were studied in \cite{Ben}.
\smallskip

The Medina-Revoy theorem and the double extension give a systematic way to study quadratic non-semisimple Lie algebras (see \cite{Med}). The Medina-Revoy theorem proves that an indecomposable quadratic non-simple Lie algebra can be obtained by a double extension of a quadratic Lie algebra of smaller dimension (see \cite{Med}). Another approach appears in \cite{Kath}, where the authors define two canonical ideals for a quadratic Lie algebra without simple ideals and construct certain cohomology set to obtain classification results. In addition, in conformal field theory, a Sugawara construction exists for quadratic Lie algebras (see \cite{Figueroa}).
\smallskip

In this work we assume that the twist map $T$, of a quadratic Hom-Lie algebra $(\g,[\,\cdot\,,\,\cdot\,],T,B)$, belongs to the \emph{centroid}, that is: $T([x,y])=[T(x),y]$, for all $x,y$ in $\g$. In \cite{Central}, is obtained a characterization for these Hom-Lie algebras, which are not Lie algebras, by using central extensions of classical quadratic Lie algebras (see \textbf{Thm. 2.2} in \cite{Central}). 
\smallskip

For classical quadratic Lie algebras, the space generated by self-adjoint centroids plays an important role to knowing the algebraic structure. If this vector space is one-dimensional then the Lie algebra is simple (see \cite{Bajo 2}). In \cite{Bajo}, it is shown that if this vector space is two-dimensional, then the Lie algebra has only one maximal ideal.
\smallskip

Quadratic Hom-Lie algebras with twist map in the centroid solve the following problem: \emph{Under what conditions a central extension of a quadratic Lie algebra $\g$ admits an invariant metric?}. The answer given in \cite{Central} is that such a central extension admits an invariant metric if and only if $\g$ admits a quadratic Hom-Lie algebra structure with twist map in the centroid (see \textbf{Thm. 1.2} and \textbf{Thm. 3.1} in \cite{Central}).
\smallskip

In \cite{Central}, some examples of such Hom-Lie algebras that are not Lie algebras are given. There seems to be interest in obtaining similar examples for other types of Hom-Lie algebras as well (see \cite{Das}, \textbf{Example} 2.4). 
\smallskip

At this point, the question arises whether there is an inductive way to construct such quadratic Hom-Lie algebras, similar to the procedure called ``double extension". If so, another question that arises is whether any \emph{indecomposable} quadratic Hom-Lie algebra with twist map in the centroid, can be obtained by means of this type of double extension, similar to Medina-Revoy theorem. In this work we prove that both questions have an affirmative answer (see \textbf{Thm. \ref{teorema}} and \textbf{Cor. \ref{corolario}}). 
\smallskip

To deal with Hom-Lie algebras that are not Lie algebras, we use the Fitting's lemma to show that the twist map $T$ is nilpotent (see \textbf{Lemma \ref{T es nilpotente}}). Next we use some results of \cite{Central}, to prove that if the bracket $[\,\cdot\,,\,\cdot\,]$ does not satisfy the classical Jacobi identity for Lie algebras, then $\g=[\g,\g]$ (see \textbf{Prop. \ref{g es perfecta}}). We will show that there exists a maximal ideal $\mathcal{I}$ containing the kernel and the image of $T$, and for any subspace $\s$ satisfying $\g=\s \oplus \mathcal{I}$, the bracket $[\,\cdot\,,\,\cdot\,]$ of $\g$ induces a simple Lie algebra structure on $\s$ (see \textbf{Prop. \ref{ideal maximal}}). 
\smallskip

Unlike double extension for classical Lie algebras, the subspace $\s$ is not a subalgebra of $\g$ (see \textbf{Prop. A.1} in \cite{Figueroa} and \textbf{Example} in \S 3 below). However we can use the result in \textbf{Prop. A.1} of \cite{Figueroa} to simplify some results obtained here, as follows: We construct a Lie algebra in the space $\p=\s \oplus T(\mathcal{I})$, in such a way that $\p$ is the semidriect sum of $\s$ and $T(\mathcal{I})$. This implies that if $\Lambda:\s \times \s \to \mathcal{I}$ is the bilinear map for which $[x,y]=[x,y]_{\s}+\Lambda(x,y)$, where $[x,y]_{\s}$ is in $\s$, for all $x,y$ in $\s$, then $T \circ \Lambda=0$ (see \textbf{Lemma \ref{T y Lambda}}). This result simplifies some data needed to state the double extension procedure for these Hom-Lie algebras. In \S 3, we provide an example of a quadratic Hom-Lie algebra with twist map in the centroid, which is not a Lie algebra, and $\s$ is not a subalgebra. All the vector spaces considered in this work are finite dimensional over a unique field $\F$ of zero characteristic.

\section{Quadratic Hom-Lie algebras with twist maps in the centroid}

Let $(\g,[\,\cdot\,,\,\cdot\,],T)$ be a Hom-Lie algebra over a field $\F$. We denote by $\ad(x):\g \to \g$, the map $\ad(x)(y)=[x,y]$, for all $x,y$ in $\g$. A subspace $I \subset \g$ is an \emph{ideal} of $\g$ if $T(I) \subset I$ and $[\g,I] \subset I$. 
\smallskip

A linear map $D:\g \to \g$ is a \emph{derivation} if $D([x,y])=[D(x),y]+[x,D(y)]$, for all $x,y$ in $\g$. We denote by $\Der(\g)$, the vector space generated by derivations of $\g$.
\smallskip

Let $B$ be a bilinear form on $\g$. A linear map $f:\g \to \g$, is \emph{skew-symmetric} with respect to $B$, if $B(f(x),y)=-B(x,f(y))$, for all $x,y$ in $\g$. We denote by $\mathfrak{o}(B)$ the vector space generated by skew-symmetric maps of $\g$ with respect to $B$..
\smallskip

If $(\g,[\,\cdot\,,\,\cdot\,],T)$ is a Hom-Lie algebra with twist map $T$ in the centroid, then the cyclic sum $[x,[y,z]]+[y,[z,x]]+[z,[x,y]]$ belongs to $\Ker(T)$, for all $x,y,z$ in $\g$. Thus if $T$ is invertible then $(\g,[\,\cdot\,,\,\cdot\,])$ is a Lie algebra. In the next result we prove that for Hom-Lie algebras thT are not Lie algebras, the twist map must be nilpotent. 

\begin{Prop}\label{T es nilpotente}{\sl
Let $(\g,[\,\cdot\,,\,\cdot\,],T)$ be a Hom-Lie algebra
with twist map $T$ in the centroid. Then $\g$ is the direct sum of a Lie algebra and a Hom-Lie algebra with nilpotent twist map in the centroid.}
\end{Prop}
\begin{proof}
By Fitting's lemma, there exists an integer $\ell$ such that: 
\begin{equation}\label{suma directa}
\g=\Im(T^{\ell}) \oplus \Ker(T^{\ell}),
\end{equation}
where $\Im(T^{\ell})=\Im(T^{\ell+k})$ and $\Ker(T^{\ell})=\Ker(T^{\ell+k})$ for all $k$ in $\N$. 
Due to $T$ is in the centroid, $\Im(T^{\ell})$ and $\Ker(T^{\ell})$ are ideals of $\g$. 
\smallskip

Since $T\vert_{\Im\left(T^{\ell}\right)}$ is invertible, 
the Hom-Lie algebra structure in $\Im(T^{\ell})$, induced by $\g$, 
is that of a Lie algebra.
On the other hand, $T\vert_{\Ker\left(T^{\ell}\right)}$ is nilpotent, thus from \eqref{suma directa} we deduce that $\g$ is the direct sum of the Lie algebra $\Im(T^{\ell})$ and the Hom-Lie algebra $(\Ker(T^{\ell}),[\,\cdot\,,\,\cdot\,],T\vert_{\Ker(T^{\ell})})$ with nilpotent twist map $T\vert_{\Ker(T^{\ell})}$ in the centroid. Further, if $\g$ admits an invariant metric, the decomposition \eqref{suma directa} corresponds to an orthogonal direct sum of $\Im(T^{\ell})$ and $\Ker(T^{\ell})$. 
\end{proof}

By \textbf{Prop. \ref{T es nilpotente}}, we assume that the twist map $T$ is nilpotent. If $\g$ admits an invariant metric and $[\,\cdot\,,\,\cdot\,]$ does not satisfy the Jacobi identity for Lie algebras, we can make another simplification.

\begin{Prop}\label{g es perfecta}{\sl
Let $(\g,[\,\cdot\,,\,\cdot\,],T,B)$ be a quadratic Hom-Lie algebra with nilpotent twist map $T$ in the centroid. If $[\,\cdot\,,\,\cdot\,]$ does not satisfy the Jacobi identity, then $\g=[\g,\g]$.}
\end{Prop}
\begin{proof}
Let $r=\dim(\Ker(T))>0$, and consider the bilinear map $[\,\cdot\,,\,\cdot\,]_{\text{Lie}}$ in $\g$ defined by $[x,y]_{\text{Lie}}=T([x,y])$, for all $x,y$ in $\g$. Then $(\g,[\,\cdot\,,\,\cdot\,]_{\text{Lie}},B)$ is a quadratic Lie algebra (see \textbf{Prop. 1.1} of \cite{Central}). 
\smallskip

Let $\a=\Span_{\F}\{a_1,\ldots,a_r\}$ be an $r$-dimensional subspace such that $\g=\a \oplus \Im(T)$. For each $i$, let $D_i:\g \to \g$ be defined by $D_i(x)=[a_i,x]$. Then $D_i$ is a derivation of $(\g,[\,\cdot\,,\,\cdot\,]_{\text{Lie}})$ (see \textbf{Prop. 1.3} in \cite{Central}).
\smallskip

By hypothesis, the bracket $[\,\cdot\,,\,\cdot\,]$ does not satisfy the Jacobi identity. Under such assumption, \textbf{Prop. 1.3} of \cite{Central}, proves that for each $i$, there exists an element $x_i$ in $\g$ such that the derivations $D^{\prime}_i=D_i+\ad_{\text{Lie}}(x_i)$, satisfy $\Ker(D^{\prime}_1)\cap \ldots \cap \Ker(D^{\prime}_r)=\{0\}$. On the other hand, due to $T$ is in the centroid of $[\,\cdot\,,\,\cdot\,]$, we have: $\ad_{\text{Lie}}(x_i)=[T(x_i),\,\cdot\,]$, for all $i$.
\smallskip

If $x$ satisfies $[x,\g]=\{0\}$, then $D^{\prime}_i(x)=[a_i,x]+[T(x_i),x]=0$ for all $i$, hence $x$ belongs to $\Ker(D^{\prime}_1)\cap \ldots \cap \Ker(D^{\prime}_r)=\{0\}$. Thus the Hom Lie algebra $(\g,[\,\cdot\,,\,\cdot\,],T)$ has trivial center. Since $B$ is non-degenerate and invariant under $[\,\cdot\,,\,\cdot\,]$, it follows that $[\g,\g]^{\perp}=\{0\}$, and $\g=[\g,\g]$. 
\end{proof}

In the next result we prove that the Hom-Lie algebra $\g$ has a maximal ideal containing the kernel and the image of the twist map.

\begin{Prop}\label{ideal maximal}{\sl
Let $(\g,[\,\cdot\,,\,\cdot\,],T)$ be a Hom-Lie algebra with nilpotent twist map $T$ in the centroid. Then:
\smallskip

\textbf{(i)} There is a maximal ideal $\mathcal{I}$ of $\g$ such that $\Ker(T)+\Im(T) \subset \mathcal{I}$.
\smallskip

\textbf{(ii)} If $\,\g$ admits an invariant metric and $[\,\cdot\,,\,\cdot\,]$ does not satisfy the Jacobi identity, then the Hom-Lie algebra structure in $\g/\mathcal{I}$, induced by $\g$, is of a simple Lie algebra.}
\end{Prop}
\begin{proof} 
$\,$

\textbf{(i)} Observe that $\g \neq \Ker(T)+\Im(T)$, because $T$ is nilpotent. Let $\mathcal{F}\!=\!\{I \subsetneq \g\mid I \text{ is an ideal containing } \Ker(T)+\Im(T)\}$. Then $\mathcal{F}$ is not empty, being that $\Ker(T)+\Im(T)$ belongs to it. A maximal element $\mathcal{I}$ of $\mathcal{F}$ satisfies the conditions required.
\smallskip

\textbf{(ii)} Suppose $\g$ admits an invariant metric and $[\,\cdot\,,\,\cdot\,]$ does not satisfy the Jacobi identity. By \textbf{Prop. \ref{g es perfecta}}, $\g=[\g,\g]$. If $\dim(\g/\mathcal{I})=1$, then $\mathcal{I}^{\perp}$ is a one-dimensional ideal contained in $\Ker(T)\cap \Im(T)$, for $\mathcal{I} \supset \Ker(T)+\Im(T)$. Let $\mathcal{I}^{\perp}=\F \,T(X)$, where $X$ is in $\g$. There is a map $\omega:\g \to \F$ such that $[T(X),Y]=\omega(Y)T(X)$, for all $Y$ in $\g$. Using the Hom-Lie Jacobi identity and that $T$ is in the centroid, we have:
$$
\aligned
& \omega([Y,Z])T(X)=[T(X),[Y,Z]]=[[T(X),Y],Z]+[Y,[T(X),Z]]\\
&=\omega(Y)\omega(Z)T(X)-\omega(Z)\omega(Y)T(X)=0,\,\text{ for all }Y,Z \in\g.
\endaligned
$$
Therefore $\omega([\g,\g])=\{0\}$ and $\omega=0$, because $\g=[\g,\g]$. Then $[T(X),\g]=\{0\}$, but $\g$ has trivial center, hence $\dim(\g/\mathcal{I})>1$.
\smallskip

Let $\s$ be a subspace of $\g$ such that $\g=\s \oplus \mathcal{I}$. For each pair $x,y$ in $\s$, there are $[x,y]_{\s}$ in $\h$ and $\Lambda(x,y)$ in $\mathcal{I}$ such that:
\begin{equation}\label{2}
[x,y]=[x,y]_{\s}+\Lambda(x,y).
\end{equation} 
Let $[\,\cdot\,,\,\cdot\,]_{\s}:\s \times \s \to \s$ be the map $(x,y)\mapsto [x,y]_{\s}$, and $\Lambda:\s \times \s \to \mathcal{I}$ the map $(x,y) \mapsto \Lambda(x,y)$. We shall prove that $(\s,[\,\cdot\,,\,\cdot\,]_{\s})$ is a simple Lie algebra. Indeed, by \eqref{2}, we have:
\begin{equation}\label{3}
[x,[y,z]]=[x,[y,z]_{\s}]_{\s}\,\,\operatorname{mod}\, \mathcal{I},\quad \text{ for all }x,y,z \in \s.
\end{equation}
The Hom-Lie Jacobi identity and the fact that $T$ is in the centroid imply that $[x,[y,z]]+[y,[z,x]]+[z,[x,y]]$ belongs to $\Ker(T) \subset \mathcal{I}$, then by \eqref{3} we deduce:
$$
[x,[y,z]_{\s}]_{\s}+[y,[z,x]_{\s}]_{\s}+[z,[x,y]_{\s}]_{\s} \in \s \cap \mathcal{I}=\{0\},
$$
from it follows that $(\s,[\,\cdot\,,\,\cdot\,]_{\s})$ is a Lie algebra. 
\smallskip

Let $J \neq \{0\}$ be a subspace of $\s$ such that $[\s,J]_{\s} \subset J$. From \eqref{2} it follows $[\g,J \oplus \mathcal{I}]\subset J \oplus \mathcal{I}$. Since $\Im(T) \subset \mathcal{I}$, then $T(J \oplus \mathcal{I})\subset \Im(T)\subset J \oplus \mathcal{I}$. Hence $J \oplus \mathcal{I}$ is an ideal containing $\mathcal{I}$. Due to $\mathcal{I}$ is maximal and $J \neq \{0\}$, $J \oplus \mathcal{I}\!=\!\g$. As $J\!\subset \!\s$, then $J=\s$ and $(\s,[\,\cdot\,,\,\cdot\,]_{\s})$ is a simple Lie algebra. 
\smallskip

It is clear that $\g/\mathcal{I}$ with bracket $(x+\mathcal{I},y+\mathcal{I}) \mapsto [x,y]+\mathcal{I}$, is isomorphic to $(\s,[\,\cdot\,,\,\cdot\,]_{\s})$.
\end{proof}

Let $(\g,[\,\cdot\,,\,\cdot\,],T,B)$ be a quadratic Hom-Lie algebra with nilpotent twist map $T$ in the centroid. We assume that $[\,\cdot\,,\,\cdot\,]$ does not satisfy the classical Jacobi identity. By \textbf{Prop. \ref{ideal maximal}}, there exists a maximal ideal $\mathcal{I}$, for which $\Ker(T)+\Im(T) \subset \mathcal{I}$ and $\g/\mathcal{I}$ is a simple Lie algebra. Consider the bilinear map $\Lambda:\s \times \s \to \mathcal{I}$ given in \eqref{2}. Although $\s$ is not necessarily a subalgebra of $\g$ (see \textbf{Example} in \S 3), the fact that $\s$ is a simple Lie algebra leads to $T \circ \Lambda=0$.

\begin{Lemma}\label{T y Lambda}{\sl
The twist map $T$ and the bilinear map $\Lambda:\s \times \s \to \mathcal{I}$ satisfy $T \circ \Lambda=0$, that is $\Lambda(x,y)$ belongs to $\Ker(T)$ for all $x,y$ in $\s$.}
\end{Lemma}
\begin{proof}
Let $\p=\s \oplus T(\mathcal{I})$. We define a bracket $[\,\cdot\,,\,\cdot\,]_{\p}$ in $\p$ by:
\begin{equation}\label{corchete de p}
\begin{split}
x,y \in \s,\,\quad & [x,y]_{\p}=[x,y]_{\s}+T(\Lambda(x,y)),\\
x \in \s,\,\,\,U \in \mathcal{I}, \quad & [x,T(U)]_{\p}=T([x,U]),\\
U,V \in \mathcal{I}, \quad & [T(U),T(V)]_{\p}=T([U,V]).
\end{split}
\end{equation}
We claim that $(\p,[\,\cdot\,,\,\cdot\,]_{\p})$ is a Lie algebra and the linear map $\Psi:\g \to \p$, $x+U \mapsto x+T(U)$, is a surjective algebra morphism with kernel $\Ker(T)$. Let $x,y$ be in $\s$ and $U,V$ be in $\mathcal{I}$; by \eqref{2} and \eqref{corchete de p} we have:
\begin{equation}\label{m2}
\begin{split}
& \Psi([x+U,y+V])\\
\,&=[x,y]_{\s}+T(\Lambda(x,y))+T([x,V])+T([U,y])+T([U,V])\\
\,&=[x,y]_{\p}+[x,T(V)]_{\p}+[T(U),y]_{\p}+[T(U),T(V)]_{\p}\\
\,&=[x+T(U),y+T(V)]_{\p}=[\Psi(x+U),\Psi(y+V)]_{\p}.
\end{split}
\end{equation}
Then $\Psi$ is a morphism of algebras. It is clear that $\Psi$ is surjective. 
\smallskip

Let $x+U$ be in $\Ker(\Psi)$, then $x+T(U)=0$. Since $\Im(T) \subset \mathcal{I}$ and $\s \cap \mathcal{I}=\{0\}$, then $x=0$ and $T(U)=0$, hence $x+U=U$ belongs to $\Ker(T)$, thus $\Ker(\Psi) \subset \Ker(T)$. 
\smallskip

If $x+U$ is in $\Ker(T)$, then $x$ is in $\Ker(T)+\mathcal{I}=\mathcal{I}$, so $x=0$, as $\s \cap \mathcal{I}=\{0\}$ and $\Ker(T) \subset \mathcal{I}$. Thus $U$ belongs to $\Ker(T)$ and $\Psi(U)=T(U)=0$. Then $\Ker(T)\subset \Ker(\Psi)$ and $\Ker(T)=\Ker(\Psi)$.
\smallskip

As $T$ is in the centroid, then $[X,[Y,Z]]+[Y,[Z,X]]+[Z,[X,Y]]$ belongs to $\Ker(T)=\Ker(\Psi)$. Due to $\Psi$ is a morphism of algebras, we get:
\begin{equation*}\label{Dom26-1}
[\Psi(X),[\Psi(Y),\Psi(Z)]_{\p}]_{\p}=\Psi([X,[Y,Z]])\,\text{ for all }X,Y,Z \in \g.
\end{equation*}
From it follows that $(\p,[\,\cdot\,,\,\cdot\,]_{\p})$ is a Lie algebra, as $\Psi$ is surjective. We claim that $T \circ \Lambda=0$. Let $x$ be in $\s$ and $U$ be in $\mathcal{I}$. The map $\Upsilon:\p \to \s$, defined by $\Upsilon(x+T(U))=x$, is a surjective Lie algebra morphism. Then $\s$ is a subalgebra of $\p$ (see \cite{Figueroa}, \textbf{Prop. A.1}). By \eqref{corchete de p}, it follows that $[\s,\s]_{\p}=[\s,\s]_{\s}+T(\Lambda(\s,\s))\subset \s$, then $T(\Lambda(\s,\s))$ lies in $\s \cap \Im(T)\subset \s \cap \mathcal{I}=\{0\}$, hence $T \circ \Lambda=0$. 

\end{proof}

\section{Indecomposable quadratic Hom-Lie algebras with nilpotent twist map in the centroid}

A quadratic Hom-Lie algebra $\g,$ is decomposable if there is a proper ideal $I \neq \{0\}$ such that $\g=I \oplus I^{\perp}$, otherwise it is indecomposable. 
\smallskip

From now on we assume that $(\g,[\,\cdot\,,\,\cdot\,],T,B)$ is an indecomposable quadratic Hom-Lie algebra, with nilpotent twist map $T$ in the centroid, and $[\,\cdot\,,\,\cdot\,]$ does not satisfy the Jacobi identity for Lie algebras. 
\smallskip

By \textbf{Prop. \ref{ideal maximal}}, there is a maximal ideal $\mathcal{I}$ such that $\g/\mathcal{I}$ is a simple Lie algebra and $\Ker(T)+\Im(T)\!\subset \!\mathcal{I}$. Observe that $\mathcal{I}\cap \mathcal{I}^{\perp}\! \neq \!0$, as $\g$ is indecomposable; thus $\mathcal{I}^{\perp} \subset \mathcal{I}$, because $\mathcal{I}^{\perp}$ is minimal. 
\smallskip

Let $\h$ be a subspace such that $\mathcal{I}=\h \oplus \mathcal{I}^{\perp}$. Then $\h$ is non-degenerate subspace, $\g\!=\!\h \oplus \h^{\perp}$ and $\mathcal{I}^{\perp} \!\subset \!\h^{\perp}$. Let $\s$ be a subspace such that $\h^{\perp}\!=\!\s \oplus \mathcal{I}^{\perp}$, then $\g\!=\!\s \oplus \h \oplus \mathcal{I}^{\perp}$. Since $B$ is non-degenerate, by the Witt decomposition we may assume that $\s$ is isotropic. We know that $\s$ is the underlying vector space of a simple Lie algebra $(\s,[\,\cdot\,,\,\cdot\,]_{\s})$ for which the Lie algebra $\p\!=\!\s \oplus T(\mathcal{I})$ is the semidirect sum of $\s$ and $T(\mathcal{I})$. This leads us to $T \circ \Lambda=0$, where $[x,y]=[x,y]_{\s}+\Lambda(x,y)$, $[x,y]_{\s}$ belongs to $\s$ and $\Lambda(x,y)$ belongs to $\mathcal{I}$, for all $x,y$ in $\s$ (see \textbf{Lemma \ref{T y Lambda}}).
\smallskip

Due to $\Lambda(x,y)$ belongs to $\mathcal{I}=\h \oplus \mathcal{I}^{\perp}$, then there are $\lambda(x,y)$ in $\h$ and $\mu(x,y)$ in $\mathcal{I}^{\perp}$ such that: $\Lambda(x,y)=\lambda(x,y)+\mu(x,y)$. Let $\lambda:\s \times \s \to \h$ be the bilinear map $(x,y) \mapsto \lambda(x,y)$, and $\mu:\s \times \s \to \mathcal{I}^{\perp}$ the bilinear map $(x,y) \mapsto \mu(x,y)$, for all $x,y$ in $\s$.
\smallskip

Let $u$ be in $\h$; since $[x,u]$ belongs to $\mathcal{I}=\h \oplus \mathcal{I}^{\perp}$, there are $\bar{\rho}(x,u)$ in $\h$ and $\bar{\tau}(x,u)$ in $\mathcal{I}^{\perp}$, such that: $[x,u]=\bar{\rho}(x,u)+\bar{\tau}(x,u)$. Let $\rho:\s \to \gl(\h)$ be the map defined by $\rho(x)(u)=\bar{\rho}(x,u)$ and $\tau:\s \to \Hom(\h;\mathcal{I}^{\perp})$ the map defined by $\tau(x)(u)=\bar{\tau}(x,u)$.
\smallskip

Let $\alpha$ be in $\mathcal{I}^{\perp}$. Due to $\mathcal{I}^{\perp}$ is an ideal of $\g$, there is a map $\sigma:\s \to \gl(\mathcal{I}^{\perp})$, such that $[x,\alpha]=\sigma(x)(\alpha)$.
\smallskip

Let $u,v$ be in $\h$. Due to $[u,v]$ belongs to $\mathcal{I}=\h \oplus \mathcal{I}^{\perp}$, there are $[u,v]_{\h}$ in $\h$ and $\gamma(u,v)$ in $\mathcal{I}^{\perp}$ such that: $[u,v]=[u,v]_{\h}+\gamma(u,v)$. Let $[\,\cdot\,,\,\cdot\,]_{\h}:\h \times \h \to \h$ be the bilinear map $(u,v) \mapsto [u,v]_{\h}$, and $\gamma:\h \times \h \to \mathcal{I}^{\perp}$ the bilinear map $(u,v) \mapsto \gamma(u,v)$.
\smallskip

Therefore, the maps $\lambda,\mu,\rho,\tau,\sigma,[\,\cdot\,,\,\cdot\,]_{\h}$ and $\gamma$ describe the bracket $[\,\cdot\,,\,\cdot\,]$ of $\g=\s \oplus \h \oplus \mathcal{I}^{\perp}$, as follows:
\begin{equation}\label{descomposicion de corchete 2}
\begin{split}
x,y \in \s,\,\quad & [x,y]=[x,y]_{\s}+\lambda(x,y)+\mu(x,y),\\
x \in \s,\,u \in \h,\quad & [x,u]=\rho(x)(u)+\tau(x)(u),\\
x \in \s,\,\alpha \in \mathcal{I}^{\perp},\quad & [x,\alpha]=\sigma(x)(\alpha),\\
u,v \in \h,\quad & [u,v]=[u,v]_{\h}+\gamma(u,v).
\end{split}
\end{equation} 
Since $\mathcal{I}^{\perp} \subset \mathcal{I}$ then $[\mathcal{I},\mathcal{I}^{\perp}]=\{0\}$, as $B$ is invariant and non-degenerate. In particular, $[\h,\mathcal{I}^{\perp}]=\{0\}$, and $\mathcal{I}^{\perp}$ is abelian.

\begin{Lemma}\label{el map xi}{\sl
There is a bijective linear map $\xi:\mathcal{I}^{\perp} \to \s^{\ast}$ such that $\xi \circ \sigma(x)=\ad_{\s}(x)^{\ast}\circ \xi$, for all $x$ in $\s$.}
\end{Lemma}
\begin{proof}
We define $\xi$ by $\xi(\alpha)(x)=B(\alpha,x)$, for all $\alpha$ in $\mathcal{I}^{\perp}$ and $x$ in $\s$. The map $\xi$ is bijective, as $B$ is non-degenerate. Let $x,y$ be in $\s$ and $\alpha$ be in $\mathcal{I}^{\perp}$. Using that $B$ is invariant and $\mathcal{I}^{\perp} \subset \mathcal{I}$, we get:
$$
\aligned
& \xi(\sigma(x)(\alpha))(y)=B([x,\alpha],y)=-B(\alpha,[x,y])\\
&=-B(\alpha,[x,y]_{\s})=-\xi(\alpha)(\ad_{\s}(x)(y))=\ad_{\s}^{\ast}(x)\circ \xi(\alpha)(y).
\endaligned
$$
Hence, $\xi \circ \sigma(x)=\ad_{\s}^{\ast}(x)\circ \xi$. Then 
$\sigma:\s \to \gl(\mathcal{I}^{\perp})$ is a representation isomorphic to the co-adjoint representation $\ad_{\s}^{\ast}:\s \to \gl(\s^{\ast})$. 
\end{proof}

Let $B_{\h}=B\vert_{\h \times \h}$, then $B_{\h}$ is a non-degenerate bilinear form on $\h$, as $\h$ is a non-degenerate subspace of $\g$. Thus, the invariant metric $B$ has the following description according to the decomposition $\g=\s \oplus \h \oplus \mathcal{I}^{\perp}$:
\begin{equation}\label{descomposicion de metrica}
B(x+u+\alpha,y+v+\beta)=\xi(\alpha)(y)+\xi(\beta)(x)+B_{\h}(u,v),
\end{equation}
for all $x,y$ in $\s$, $u,v$ in $\h$ and $\alpha,\beta$ in $\mathcal{I}^{\perp}$. 
\smallskip

Let $\mu_{\xi}=\xi \circ \mu:\s \times \s \to \s^{\ast}$. Using that $B$ is invariant we get:
\begin{equation}\label{mu es ciclica}
\begin{split}
& B(x,[y,z])=B(x,\mu(y,z))=\xi(\mu(y,z))(x)=\mu_{\xi}(y,z)(x)\\
&=B([x,y],z)=\mu_{\xi}(x,y)(z),\quad \text{ for all }x,y,z in \s.
\end{split}
\end{equation}
Then $\mu_{\xi}(x,y)(z)=\mu_{\xi}(y,z)(x)$. For each $x$ in $\s$, define the map $\tau_{\xi}(x)=\xi \circ \tau(x):\h \to \s^{\ast}$. As $B$ is invariant under $[\,\cdot\,,\,\cdot\,]$, we get:
$$
\aligned
& B([x,u],y)=B(\tau(x)(u),y)=\xi(\tau(x)(u))(y)=\tau_{\xi}(x)(u)(y)\\
&=-B([x,y],u)=-B_{\h}(\lambda(x,y),u).
\endaligned
$$
Then, 
\begin{equation}\label{Dom26Mayo}
\tau_{\xi}(x)(u)(y)=-B_{\h}(\lambda(x,y),u),\,\text{ for all }x,y \in \s\,\text{ and }u \in \h.
\end{equation}

Define the map $\gamma_{\xi}=\xi \circ \gamma:\s \times \s \to \s^{\ast}$. Since $B$ is invariant we get:
$$
\aligned
& B([x,u],v)=B_{\h}(\rho(x)(u),v)=B(x,[u,v])\\
&=B(x,\gamma(u,v))=\xi(\gamma(u,v))(x)=\gamma_{\xi}(u,v)(x).
\endaligned
$$
Hence,
\begin{equation}\label{gamma y rho 1}
B_{\h}(\rho(x)(u),v)=\gamma_{\xi}(u,v)(x),\,\,\text{ for all }u,v \in \h,\text{ and }x \in \s.
\end{equation}

Now we shall give a description of the twist map $T$ according to the decomposition $\g=\s \oplus \h \oplus \mathcal{I}^{\perp}$. Let $x$ be in $\s$, since $\Im(T) \subset \mathcal{I}$ (see \textbf{Lemma \ref{ideal maximal}}), then $T(x)$ lies in $\mathcal{I}=\h \oplus \mathcal{I}^{\perp}$. Hence there are $\phi(x)$ in $\h$ and $\varphi(x)$ in $\mathcal{I}^{\perp}$ such that $T(x)=\phi(x)+\varphi(x)$. Let $\phi:\s \to \h$ be the map $x \mapsto \phi(x)$, and $\varphi:\s \to \mathcal{I}^{\perp}$ the map $x \mapsto \varphi(x)$. 
\smallskip

Let $u$ be in $\h$. Since $T(u)$ lies in $I=\h \oplus \mathcal{I}^{\perp}$, there are $\Theta(u)$ in $\h$ and $L(u)$ in $\mathcal{I}^{\perp}$ such that $T(u)=\Theta
(u)+L(u)$. Let $\Theta:\h \to \h$ be the map $u \mapsto \Theta(u)$, and $L:\h \to \mathcal{I}^{\perp}$ the map $u \mapsto L(u)$. 
\smallskip

Due to $\mathcal{I}^{\perp} \subset \Ker(T)$, the twist map $T$ has the following description:
\begin{equation}\label{descomposicion de T}
\begin{split}
x \in \s, \quad & T(x)=\phi(x)+\varphi(x),\\
u \in \h, \quad & T(u)=\Theta(u)+L(u).
\end{split}
\end{equation}
Define the map $\varphi_{\xi}\!=\!\xi \circ \!\varphi:\s \!\to \!\s^{\ast}$. As $T$ is self-adjoint, by \eqref{descomposicion de metrica} we get:
\begin{equation}\label{sf2}
\varphi_{\xi}(x)(y)=B(\varphi(x),y)=B(x,\varphi(y))=\varphi_{\xi}(y)(x),\tag{\textbf{A}}
\end{equation}
for all $x,y$ in $\s$. Similarly, define $L_{\xi}=\xi \circ L:\h \to \s^{\ast}$, then
\begin{equation}\label{sf0}
L_{\xi}(u)(x)\!=\!B(L(u),x)\!=\!B_{\h}(u,\phi(x)),\,\text{ for all }x \in \s\,\text{ and }u \in \h.
\end{equation}

\begin{Lemma}\label{h is cuadratica}{\sl
The tuple $(\h,[\,\cdot\,,\,\cdot\,]_{\h},\Theta,B_{\h})$ is a quadratic Hom-Lie algebra with nilpotent twist map $\Theta$ in the centroid.}
\end{Lemma}
\begin{proof}
By \eqref{descomposicion de T}, $\Theta:\h \to \h$ is a nilpotent map, because $T$ is nilpotent and $\mathcal{I}^{\perp} \subset \Ker(T)$. Let $u,v,w$ be in $\h$. Comparing the components in $\h$ and $\mathcal{I}^{\perp}$ of $T([u,v])$ and $[T(u),v]$, by \eqref{descomposicion de corchete 2} and \eqref{descomposicion de T}, we deduce that $T([u,v])\!=\![T(u),v]$ is equivalent to:
\begin{align}
\label{T1} \,& \Theta([u,v]_{\h})=[\Theta(u),v]_{\h},\,\,\,\text{ and }\\
\label{T2} \,& L([u,v]_{\h})=\gamma(\Theta(u),v).
\end{align}
The expression \eqref{T1} says that $\Theta$ is in the centroid of $(\h,[\,\cdot\,,\,\cdot\,]_{\h})$. 
\smallskip

By \eqref{descomposicion de corchete 2}, it follows that $[u,[v,w]]=[u,[v,w]_{\h}]_{\h} \operatorname{mod} \mathcal{I}^{\perp}$. We apply $T$ and we use the facts that $\Theta$ is in the centroid of $[\,\cdot\,,\,\cdot\,]_{\h}$, $\mathcal{I}^{\perp} \subset \Ker(T)$ and that $[\h,\mathcal{I}^{\perp}]=\{0\}$, to get:
$$
[T(u),[v,w]]=T([u,[v,w]])=[\Theta(u),[v,w]_{\h}]_{\h}\,\operatorname{mod} \mathcal{I}^{\perp}.
$$
Due to $[T(u),[v,w]]+[T(v),[w,u]]+[T(w),[u,v]]=0$, then:
$$
[\Theta(u),[v,w]_{\h}]_{\h}+[\Theta(v),[w,u]_{\h}]_{\h}+[\Theta(w),[u,v]_{\h}]_{\h} \in \h \cap \mathcal{I}^{\perp}=\{0\}.
$$
Hence, $(\h,[\,\cdot\,,\,\cdot\,]_{\h},\Theta)$ is a Hom-Lie algebra. In addition,
$$
B_{\h}(\Theta(u),v)=B(T(u),v)=B(u,T(v))=B_{\h}(u,\Theta(v)),
$$
then $\Theta$ is $B_{\h}$-self adjoint. Since $B$ is invariant under $[\,\cdot\,,\,\cdot\,]$ and $B(\h,\mathcal{I}^{\perp})=\{0\}$, we obtain: $B_{\h}(u,[v,w]_{\h})\!=\!B(u,[v,w])\!=\!B([u,v],w)\!=\!B_{\h}([u,v]_{\h},w)$. Then $B_{\h}$ is invariant under $[\,\cdot\,,\,\cdot\,]_{\h}$ and $(\h,[\,\cdot\,,\,\cdot\,]_{\h},\Theta,B_{\h})$ is a quadratic Hom-Lie algebra with nilpotent twist map $\Theta$ in the centroid. 
\end{proof}

Let $x$ be in $\s$ and $u$ be in $\h$. Comparing the components in $\h$ and $\mathcal{I}^{\perp}$, from \eqref{descomposicion de corchete 2} and \eqref{descomposicion de T}, we deduce that $T([x,u])=[T(x),u]=[x,T(u)]$ is equivalent to:
\begin{align}
\label{T3} & \Theta \circ \rho(x)=\ad_{\h}(\phi(x))=\rho(x) \circ \Theta,\,\,\,\text{ and }\tag{\textbf{B}}\\
\label{T4} & L(\rho(x)(u))=\gamma(\phi(x),u)=\tau(x)(\Theta(u))+\sigma(x)(L(u)).
\end{align}

Using that $T \circ \Lambda=0$ (see \textbf{Prop. \ref{T y Lambda}}), we can simplify \eqref{T4}.

\begin{claim}\label{tau y T_h}{\sl
$\tau(x)(\Theta(u))=0$ for all $x$ in $\s$ and $u$ in $\h$.}
\end{claim}
\begin{proof}
Since $\Lambda(x,y)=\lambda(x,y)+\mu(x,y)$, where $\lambda(x,y)$ is in $\h$ and $\mu(x,y)$ is in $\mathcal{I}^{\perp} \subset \Ker(T)$, then $T(\Lambda(x,y))=T(\lambda(x,y))=0$. By \eqref{Dom26Mayo} we have:
$$
\aligned
\tau_{\xi}(x)(\Theta(u))(y)&=-B_{\h}(\lambda(x,y),\Theta(u))=-B(\lambda(x,y),T(u))\\
&=-B(T(\lambda(x,y)),u)=0.
\endaligned
$$
Then, $\tau_{\xi}(x)(\Theta(u))=\xi(\tau(x)(\Theta(u)))=0$; hence $\tau(x)(\Theta(u))=0$. Therefore, \eqref{T4} can be written as:
\begin{equation}\label{T4-1}
L(\rho(x)(u))=\gamma(\phi(x),u)=\sigma(x)(L(u)).
\end{equation}
\end{proof}

Let $x,y$ be in $\s$; comparing the components in $\h$ and $\mathcal{I}^{\perp}$, from \eqref{descomposicion de corchete 2} and \eqref{descomposicion de T}, we see that $T([x,y])=[T(x),y]=[x,T(y)]$ is equivalent to:
\begin{align}
\label{T5} & \phi([x,y]_{\s})=\rho(x)(\phi(y))=-\rho(y)(\phi(x)),\,\,\,\text{ and }\tag{\textbf{C}}\\
\label{T6} & \varphi([x,y]_{\s})=\tau(x)(\phi(y))+\sigma(x)(\varphi(y)).
\end{align}

Using that $T \circ \Lambda=0$ (see \textbf{Prop. \ref{T y Lambda}}), we can simplify \eqref{T6}.

\begin{claim}\label{tau y phi}{\sl
$\tau(x)(\phi(y))=0$ for all $x,y$ in $\s$.}
\end{claim}
\begin{proof}
Let $x,y,z$ be in $\s$. By \eqref{Dom26Mayo}, we have:
$$
\aligned
\tau_{\xi}(x)(\phi(y))(z)&=-B_{\h}(\lambda(x,z),\phi(y))=-B(\lambda(x,z),\phi(y))\\
&=-B(\lambda(x,z),T(y))=-B(T(\lambda(x,z)),y)=0.
\endaligned
$$
Then, $\tau_{\xi}(x)(\phi(y))=\xi(\tau(x)(\phi(y)))=0$; hence, $\tau(x)(\phi(y))=0$. Therefore, \eqref{T6} can be written as:
\begin{equation}\label{P}
\varphi([x,y]_{\s})=\sigma(x)(\varphi(y)),\,\,\text{ for all }x,y \in \s.\tag{\textbf{D}}
\end{equation}
\end{proof}

\begin{Lemma}\label{ciclica 1}{\sl
Let $x,y$ be in $\s$ and $u$ in $\h$. Then, $[T(x),[y,u]]+[T(y),[u,x]]+[T(u),[x,y]]=0$, is equivalent to:
\begin{align}
\label{C1} & \rho([x,y]_{\s})\circ \Theta=\rho(x)\circ \rho(y)\circ \Theta-\rho(y)\circ \rho(x)\circ \Theta,\,\text{ and }\\
\label{C2} & \rho([x,y]_{\s})(u)-\rho(x)\circ \rho(y)(u)+\rho(y)\circ \rho(x)(u) \in \Ker(L).
\end{align}
}
\end{Lemma}
\begin{proof}
As $\mathcal{I}^{\perp} \subset \Ker(T)$ and $[\h,\mathcal{I}^{\perp}]=\{0\}$, by \eqref{T3} and \eqref{T4-1} it follows:
$$
\aligned
[T(x),[y,u]]&=[T(x),\rho(y)(u)]=[\phi(x),\rho(y)(u)]\\
&=[\phi(x),\rho(y)(u)]_{\h}+\gamma(\phi(x),\rho(y)(u))\\
&=\underbrace{\rho(x)\circ \rho(y)(\Theta(u))}_{\in \h}+\underbrace{L(\rho(x)\circ \rho(y)(u))}_{\in \mathcal{I}^{\perp}}.
\endaligned
$$
Similarly: 
$$
[T(y),[u,x]]=\underbrace{-\rho(y)\circ \rho(x)(\Theta(u))}_{\in \h}-\underbrace{L(\rho(y)\circ \rho(x)(u))}_{\in \mathcal{I}^{\perp}}.
$$
Using that $T$ is in the centroid and $T \circ \Lambda=0$, we get:
$$
\aligned
[T(u),[x,y]]&=[T(u),[x,y]_{\s}]=[u,T([x,y]_{\s})]=[u,\phi([x,y]_{\s})]\\
&=[u,\phi([x,y]_{\s})]_{\h}+\gamma(u,\phi([x,y]_{\s}))\\
&=\underbrace{-\rho([x,y]_{\s})(\Theta(u))}_{\in \h}-\underbrace{L(\rho([x,y]_{\s})(u))}_{\in \mathcal{I}^{\perp}}.
\endaligned
$$
Comparing the components in $\h$ and $\mathcal{I}^{\perp}$ of $[T(x),[y,u]]$, $[T(y),[u,x]]$ and $[T(u),[x,y]]$, we obtain \eqref{C1} and \eqref{C2}.
\end{proof}

\begin{Lemma}\label{ciclica 2}{\sl
Let $x$ be in $\s$ and $u,v$ be in $\h$. Then, $[T(x),[u,v]]+[T(u),[v,x]]+[T(v),[x,u]]=0$, is equivalent to:
\begin{align}
\label{C4} & \ad_{\h}(\phi(x))([u,v]_{\h})=[\ad_{\h}(\phi(x)),u]_{\h}+[u,\ad_{\h}(\phi(x))(v)]_{\h},\,\text{ and }\\
\label{C5} & \rho(x)([u,v]_{\h})-[\rho(x)(u),v]_{\h}-[u,\rho(x)(v)]_{\h} \in \Ker(L).
\end{align}
Observe that \eqref{C4} says that $\ad_{\h}(\phi(x))$ is a derivation of $(\h,[\,\cdot\,,\,\cdot\,]_{\h})$. In addition, by \eqref{T3}, the expression \eqref{C4} can be written as:
\begin{equation}\label{C6}
\rho(x)([u,v]_{\h})-[\rho(x)(u),v]_{\h}-[u,\rho(x)(v)]_{\h} \in \Ker(\Theta).
\end{equation}}
\end{Lemma}
\begin{proof}
Using \eqref{T2}, \eqref{T3}, and the arguments of \textbf{Lemma \ref{ciclica 1}, }we get:
$$
\aligned
[T(u),[v,x]]&=-[T(u),\rho(x)(v)]=-[u,T(\rho(x)(v))]=-[u,\Theta(\rho(x)(v))]\\
\,&=-[u,\ad_{\h}(\phi(x)(v))]_{\h}-\gamma(u,\Theta(\rho(x)(v)))\\
\,&=\underbrace{-[u,\ad_{\h}(\phi(x)(v))]_{\h}}_{\in \h}-\underbrace{L([u,\rho(x)(v)]_{\h})}_{\in \mathcal{I}^{\perp}}.
\endaligned
$$
Similarly, 
$$
[T(v),[x,u]]=-\underbrace{[\ad_{\h}(\phi(x)(u)),v]_{\h}}_{\in \h}-\underbrace{L([\rho(x)(u),v]_{\h})}_{\in \mathcal{I}^{\perp}}.
$$
By \eqref{T4} we obtain:
$$
\aligned
[T(x),[u,v]]&=[T(x),[u,v]_{\h}]=[\phi(x),[u,v]_{\h}]\\
&=[\phi(x),[u,v]_{\h}]_{\h}+\gamma(\phi(x),[u,v]_{\h})\\
&=\underbrace{\ad_{\h}(\phi(x))([u,v]_{\h})}_{\in \h}+\underbrace{L(\rho(x)([u,v]_{\h}))}_{\in \mathcal{I}^{\perp}}.
\endaligned
$$
Comparing the components in $\h$ and $\mathcal{I}^{\perp}$ of $[T(u),[v,x]]$, $[T(v),[x,u]]$ and $[T(x),[u,v]]$, we obtain \eqref{C4} and \eqref{C5}.
\end{proof}

We gather the expressions \eqref{C1}, \eqref{C4} and \eqref{C6}, in the following result. 

\begin{Lemma}\label{Im T}{\sl
The map $\rho^{\prime}:\s \to \Der(\Im(\Theta))$, $x \mapsto \rho(x)\vert_{\Im(\Theta)}$ is a Lie algebra representation of $(\s,[\,\cdot\,,\,\cdot\,]_{\s})$.}
\end{Lemma}
\begin{proof}
Let $x$ be in $\s$. We know by \eqref{T3} that $\rho(x)$ and $\Theta$ commute, then $\rho(x)(\Im(\Theta)) \subset \Im(\Theta)$. Let $u,v$ be in $\h$; applying $\Theta$ to \eqref{C6}, we obtain:
\begin{equation}\label{viernes24}
\rho(x)([\Theta(u),v]_{\h})=[\rho(x)(\Theta(u)),v]_{\h}+[\Theta(u),\rho(x)(v)]_{\h}.
\end{equation}
Replace $v$ by $\Theta(v)$ in \eqref{viernes24} and observe that $\rho(x)([\Theta(u),\Theta(v)]_{\h})=\rho^{\prime}(x)([\Theta(u),\Theta(v)]_{\h})$, as $\Theta$ is in the centroid of $\h$. Then we deduce from \eqref{viernes24} that $\rho^{\prime}(x)=\rho(x)\vert_{\Im(\Theta)}$ is a derivation of $\Im(\Theta)$.
\smallskip

Let $w=\rho(x)(\Theta(u))=\Theta(\rho(x)(u))$, which belongs to $\Im(\Theta)$. Then $\rho(x)(w)=\rho^{\prime}(x)(w)=\rho^{\prime}(x)\circ \rho^{\prime}(y)(\Theta(u))$. Thus by \eqref{C1} we have:
$$
\rho^{\prime}([x,y]_{\s})(\Theta(u))=\rho^{\prime}(x)\circ \rho^{\prime}(y)(\Theta(u))-\rho^{\prime}(y)\circ \rho^{\prime}(x)(\Theta(u)),
$$
hence $\rho^{\prime}:\s \to \Der(\Im(\Theta))$ is a Lie algebra representation of $\s$. 
\end{proof}

Let us make the identifications $\mathcal{I}^{\perp}=\s^{\ast}$ and:
$$
\sigma=\ad_{\s}^{\ast},\quad \mu=\mu_{\xi},\quad \tau=\tau_{\xi},\quad  \gamma=\gamma_{\xi},\quad \mathcal{I}^{\perp}=\s^{\ast},\quad \varphi=\varphi_{\xi},\quad L=L_{\xi},
$$
induced by $\xi:\mathcal{I}^{\perp} \to \s^{\ast}$ (see \textbf{Lemma \ref{el map xi}}), in order to state a similar construction to the double extension for classical Lie algebras. 
\smallskip

We may notice that there is a lot of data to deal with because unlike classical Lie algebras, we have a linear map $T$ and the simple Lie algebra $\s$ is not a subalgebra of $\g$ (see \cite{Figueroa}, \textbf{Proposition A.1} and \textbf{Example} below). However, we have proved that $T \circ \Lambda=0$, which simplified some expressions (see \textbf{Claims \eqref{tau y T_h}} and \eqref{tau y phi}).
\begin{Theorem}\label{teorema}{\sl
Let $(\h,[\,\cdot\,,\,\cdot\,]_{\h},\Theta,B_{\h})$ be a quadratic Hom-Lie algebra with twist map $\Theta$ is in the centroid. Let $(\s,[\,\cdot\,,\,\cdot\,]_{\s})$ be a Lie algebra, and suppose that there are linear maps:
$$
\phi:\s \to \h, \quad \varphi:\s \to \s^{\ast},\quad \rho:\s \to \mathfrak{o}(B_{\h}),\quad \tau:\s \to \Hom(\h;\s^{\ast})
$$
and a skew-symmetric bilinear map $\mu:\s \times \s \to \s^{\ast}$, satisfying:
\smallskip

\textbf{(A)} $\varphi(x)(y)=\varphi(y)(x)$, for all $x,y$ in $\s$. 
\smallskip

\textbf{(B)} $\Theta \circ \rho(x)\!=\!\ad_{\h}(\phi(x))\!=\!\rho(x)\circ \Theta$ belongs to $\Der(\h)$, for all $x$ in $\s$.
\smallskip

\textbf{(C)} $\phi([x,y]_{\s})=\rho(x)(\phi(y))$, for all $x,y$ in $\s$.
\smallskip

\textbf{(D)} $\varphi([x,y]_{\s})=\ad_{\s}^{\ast}(x)(\varphi(y))$, for all $x,y$ in $\s$.
\smallskip

\textbf{(E)} The map $\rho^{\prime}:\s \to \Der(\Im(\Theta))$, $\rho^{\prime}(x)=\rho(x)\vert_{\Im(\Theta)}$, is a Lie algebra representation of $(\s,[\,\cdot\,,\,\cdot\,]_{\s})$.
\smallskip

\textbf{(F)} $\!\tau(x)\!\circ \!\Theta\!=\!\tau(x)\!\circ \!\phi\!=\!0$, and $\tau(x)(u)(x)\!=\!0$, for all $x \in \s$ and $u \in \h$.
\smallskip

\textbf{(G)} $\mu(x,y)(z)=\mu(y,z)(x)$ for all $x,y,z$ in $\s$.
\smallskip

Then in the vector space $\g=\s \oplus \h \oplus \s^{\ast}$, there exists a quadratic Hom-Lie algebra $(\g,[\,\cdot\,,\,\cdot\,],T,B)$ with twist map $T$ in the centroid.
}
\end{Theorem}

\begin{proof}

Let $L:\h \to \s^{\ast}$ be defined by $L(u)(x)=B_{\h}(u,\phi(x))$, for all $u$ in $\h$ and $x$ in $\s$ (see \eqref{sf0}). We define $T:\g \to \g$ as in \eqref{descomposicion de T}:
\begin{equation}\label{descripcion de T-teorema}
\begin{split}
x \in \s, \quad & T(x)=\phi(x)+\varphi(x),\\
u \in \h, \quad & T(u)=\Theta(u)+L(u),\\
\alpha \in \s^{\ast},\quad & T(\alpha)=0,
\end{split}
\end{equation}
Let $x,y$ be in $\s$. The map $u \mapsto -\tau(x)(u)(y)$, belongs to $\h^{\ast}$. Due to $B_{\h}$ is non-degenerate, there exists a unique element $\lambda(x,y)$ in $\h$ such that $-\tau(x)(u)(y)=B_{\h}(\lambda(x,y),u)$, for all $u$ in $\h$. Let $\lambda:\s \times \s \to \h$ be defined by $(x,y) \mapsto \lambda(x,y)$ (see \eqref{Dom26Mayo}).
\smallskip

Consider $\gamma:\h \times \h \to \s^{\ast}$ defined by: $\gamma(u,v)(x)=B_{\h}(\rho(x)(u),v)$, for all $u,v$ in $\h$ and $x$ in $\s$ (see \eqref{gamma y rho 1}). Let $[\,\cdot\,,\,\cdot\,]:\g \times \g \to \g$, be the bracket defined as in \eqref{descomposicion de corchete 2}:
\begin{equation}\label{corchete Hom-Lie T}
\begin{split}
x,y \in \s,\,\quad & [x,y]=[x,y]_{\s}+\lambda(x,y)+\mu(x,y),\\
x \in \s,\,u \in \h,\quad & [x,u]=\rho(x)(u)+\tau(x)(u),\\
x \in \s,\,\alpha \in \s^{\ast},\quad & [x,\alpha]=\ad_{\s}^{\ast}(x)(\alpha),\\
u,v \in \h,\quad & [u,v]=[u,v]_{\h}+\gamma(u,v),
\end{split}
\end{equation}
We define a non-degenerate symmetric bilinear form $B$ in $\g$, by:
\begin{equation}\label{metrica T}
B(x+u+\alpha,y+v+\beta)=\alpha(y)+\beta(x)+B_{\h}(u,v),
\end{equation}
where $x,y$ are in $\s$, $u,v$ are in $\h$ and $\alpha,\beta$ are in $\s^{\ast}$. The proof that $(\g,[\,\cdot\,,\,\cdot\,],T,B)$ is a quadratic Hom-Lie algebra with twist map $T$ in the centroid, consists only on reversing the process we did previously. For that reason we put \textbf{(A)}, \textbf{(B)}, \textbf{(C)}, \textbf{(D)} to highlight the hypothesis needed. The hypothesis in \textbf{(E)} corresponds to \textbf{Lemma \ref{Im T}}. The hypothesis in \textbf{(F)} leads to $T(\lambda(x,y))=0$ (see \textbf{Prop. \ref{T y Lambda}}). The hypothesis in \textbf{(G)} corresponds to the fact that $\mu$ is cyclic (see \eqref{mu es ciclica}).

\end{proof}

\begin{Cor}\label{corolario}{\sl
Any indecomposable quadratic Hom-Lie algebra, which is not a Lie algebra, having nilpotent twist map in the centroid, can be constructed as in \textbf{Thm. \ref{teorema}}.}
\end{Cor}
\begin{proof}
Let $\g$ be an indecomposable quadratic Hom-Lie algebra with nilpotent twist map in the centroid, such that $\g$ is not a Lie algebra. Then $\g=\s \oplus \h \oplus \mathcal{I}^{\perp}$, where $\mathcal{I}=\h \oplus \mathcal{I}^{\perp}$ is a maximal ideal of $\g$ and $\s$ is a simple Lie algebra (see \textbf{Lemma \ref{ideal maximal}}). Consider the map $\Omega:\g \to \s \oplus \h \oplus \s^{\ast}$, defined by $\Omega( x+u+\alpha)=x+u+\xi(\alpha)$, where the Hom-Lie algebra on $\s \oplus \h \oplus \s^{\ast}$ is given by \textbf{Thm. \ref{teorema}} and $\xi:\mathcal{I}^{\perp} \to \s^{\ast}$ is the map of \textbf{Lemma \ref{el map xi}}. Then $\Omega$ is an isometry. 
\end{proof}

\section{Example}

Let $\s=\mathfrak{sl}_3(\F)=\operatorname{Span}_{\F}\{e_{11}-e_{22},e_{22}-e_{33},e_{12},e_{13},e_{21},e_{23},e_{31},e_{32}\}$ be the space of $3 \times 3$ matrices over $\F$ having trace zero, where $e_{ij}$ is the matrix whose entry $(i,j)$ is equal to 1 and zero elsewhere. We consider $\s$ with the usual bracket $[x,y]_{\s}=xy-yx$, for all $x,y$ in $\s$. 
\smallskip

Let us write $x_1=e_{11}-e_{22}$, $x_2=e_{22}-e_{33}$, $x_3=e_{12}$, $x_4=e_{13}$, $x_5=e_{21}$, $x_6=e_{23}$, $x_7=e_{31}$ and $x_8=e_{32}$. 
\smallskip

Let $\mu_{ijk}$ be in $\F$ satisfying the cyclic and antisymmetric conditions: 
$$
\mu_{ijk}=\mu_{jki}=-\mu_{jik},\quad \text{ for all }1 \leq i,j,k \leq 8.
$$
Let $\s^{\ast}=\operatorname{Span}_{\F}\{\alpha_1,\ldots,\alpha_8\}$, where $\alpha_j(x_k)=\delta_{jk}$. We define the bilinear map $\mu:\s \times \s \to \s^{\ast}$, by $\mu(x_j,x_k)=\mu_{1jk}\alpha_i+\ldots+\mu_{8jk}\alpha_i$, for all $1 \leq j,k \leq 8$. Then $\mu$ is skew-symmetric and $\mu(x_j,x_k)(x_i)=\mu_{ijk}=\mu_{jki}=\mu(x_k,x_i)(x_j)$, for all $i,j,k$. In $\g=\s \oplus \s^{\ast}$, consider the bracket $[\,\cdot\,,\,\cdot\,]$ defined by: $[x+\alpha,y+\beta]=[x,y]_{\s}+\mu(x,y)+\ad_{\s}^{\ast}(x)(\beta)-\ad_{\s}^{\ast}(y)(\alpha)$, and the bilinear form $B$, by $B(x+\alpha,y+\beta)=\alpha(y)+\beta(x)$, where $x,y$ are in $\s$ and $\alpha,\beta$ are in $\s^{\ast}$. 
\smallskip

Let $K_{\s}$ be the Killing form of $\s$ and let $T:\g \to \g$ be the map defined by $T(x+\alpha)=K_{\s}(x,\,\cdot\,)$, for all $x$ in $\s$ and $\alpha$ in $\s^{\ast}$. By \textbf{Thm. \ref{teorema}}, the triple $(\g,[\,\cdot\,,\,\cdot\,],T,B)$ is a quadratic Hom-Lie algebra with twist map $T$ in the centroid. Observe that in this case, $\h=0$ and $\varphi:\s \to \s^{\ast}$, $x \mapsto K_{\s}(x,\,\cdot\,)$ satisfies \textbf{(A)} and \textbf{(D)} of \textbf{Thm. \ref{teorema}}.
\smallskip

We let $\mu_{123}=0$, $\mu_{\ell,1,3}=\mu_{\ell,2,3}=0$, for $5 \leq \ell \leq 8$ and $\mu_{234} \neq 0$, then $\g$ is not a Lie algebra. A straightforward calculation shows that:
$$
\aligned
& [x_1,[x_2,x_3]]=-2x_3-\mu_{234}\alpha_4,\\
& [x_2,[x_3,x_1]]=2x_3-2\mu_{234},\quad \text{ and }\quad [x_3,[x_1,x_2]]=0.\\
\endaligned
$$
Then, $[x_1,[x_2,x_3]]+[x_2,[x_3,x_1]]+x_3,[x_1,x_2]]=-3\mu_{234}\alpha_4 \neq 0$.

\begin{Remark}{\rm
This example shows that $\s$ is not a subalgebra of $\g$, as $\mu \neq 0$. This is a main difference from classical Lie algebras where $\s$ is always a subalgebra (see \textbf{Prop. A.1} in \cite{Figueroa}). We can carry out a similar example with $\s=\mathfrak{sl}_2$. Since $\mathfrak{sl}_2$ is the only simple Lie algebra for which $\Der(\mathfrak{sl}_2)=\mathfrak{o}(K_{\mathfrak{sl}_2})$, then we would have $[\s,\s] \subset \s$.}
\end{Remark}

\section*{Acknowledgements}

The author thanks the support provided by post-doctoral fellowship CONAHCYT 769309. The author has no conflicts to disclose.

\end{document}